\documentclass[12pt,oneside,english]{amsart}
\usepackage[T1]{fontenc}
\usepackage[latin1]{inputenc}

\usepackage{geometry}
\usepackage{fancyhdr}
\usepackage{setspace}
\usepackage{amssymb}

\makeatletter
 \theoremstyle{plain}
\newtheorem{thm}{Theorem}[section]
  \theoremstyle{remark}
  \newtheorem{rem}[thm]{Remark}
  \theoremstyle{plain}
  \newtheorem{prop}[thm]{Proposition}
  \theoremstyle{plain}
  \newtheorem{cor}[thm]{Corollary}
 \theoremstyle{definition}
  \newtheorem{example}[thm]{Example}
  \theoremstyle{plain}
  \newtheorem{lem}[thm]{Lemma}
  \newtheorem*{acknowledgement*}{Acknowledgement}

\numberwithin{equation}{section}

\usepackage{fancyhdr}
\usepackage{babel}
\makeatother
\begin{document}

\title{the central limit theorem for monotone convolution with
applications to free l\'{e}vy processes and infinite ergodic theory}

\author{jiun-chau wang}

\date{July 14, 2012; revised on February 27, 2013}

\begin{abstract}
Using free harmonic analysis and the theory of regular variation, we
show that the monotonic strict domain of attraction for the standard
arc-sine law coincides with the classical one for the standard
normal law. This leads to the most general form of the monotonic
central limit theorem and a complete description for the asymptotics
of the norming constants. These results imply that the L\'{e}vy
measure for a centered free L\'{e}vy process of the second kind
cannot have a slowly varying truncated variance. In particular, the
second kind free L\'{e}vy processes with zero means and finite
variances do not exist. Finally, the method of proofs allows us to
construct a new class of conservative ergodic measure preserving
transformations on the real line $\mathbb{R}$ equipped with Lebesgue
measure, showing an unexpected connection between free analysis and
infinite ergodic theory.
\end{abstract}

\address{Department of Mathematics and Statistics, University of Saskatchewan,
Saskatoon, Saskatchewan S7N 5E6, Canada}

\email{jcwang@math.usask.ca}



\maketitle

\section{Introduction}

Denote by $\mathbb{C}^{+}=\{ z\in\mathbb{C}:\Im z>0\}$ the complex
upper half-plane, and let
$F:\mathbb{C}^{+}\rightarrow\mathbb{C}^{+}$ be an analytic map with
$F(iy)/iy\rightarrow1$ as $y\rightarrow\infty$. This paper aims to
investigate the convergence properties of the \emph{}iterations
\emph{}$F\circ F\circ\cdots\circ F$ through free probability tools
and Karamata's theory of regular variation.

The probabilistic framework for these iterations is built upon the
theory of monotone convolution $\triangleright$, which was
introduced by Muraki in \cite{Murakipreprint,MurakiCLT} according to
his notion of monotonic independence. Denote by $\mathcal{M}$ the
set of all Borel probability measures on $\mathbb{R}$. The
convolution $\triangleright$ is an associative binary operation on
$\mathcal{M}$ that corresponds to the addition of monotonically
independent self-adjoint random variables. The monotonic
independence is one of the five natural notions of independence in
noncommutative probability
\cite{Schurmann,BenGhoralSchurmann,Speicher,MurakiFiveInd}, and
hence the corresponding monotone convolution becomes a fundamental
object in this theory. The connection between the iteration of the
function $F$ and monotone convolution is that there exists a unique
measure $\mu\in\mathcal{M}$ such that the $n$-fold iteration
$F^{\circ n}=F\circ F\circ\cdots\circ F$ of $F$ is precisely the
reciprocal Cauchy transform of the $n$-th monotone convolution power
$\mu^{\triangleright
n}=\mu\triangleright\mu\triangleright\cdots\triangleright\mu$ of
$\mu$ for $n\geq1$(see Section 2).

In addition, monotone convolution also appears in the context of
free probability theory. Indeed, by the subordination results of
Biane and Voiculescu \cite{Biane,VoiCoalgebra}, for any measures
$\rho,\tau\in\mathcal{M}$ there exist unique measures $\sigma_{1}$
and $\sigma_{2}$ in $\mathcal{M}$ such that the free convolution
$\rho\boxplus\tau=\rho\triangleright\sigma_{1}=\tau\triangleright\sigma_{2}$.
The measures $\sigma_{1}$ and $\sigma_{2}$ are interpreted in
\cite{Biane} as the Markov transitions for additive processes with
free increments, and from this perspective Biane further introduced
two natural classes of free L\'{e}vy processes: the free additive
processes with stationary increment laws (the first kind) and the
ones with stationary transition probabilities (the second kind).
These two types of stationarity condition are not equivalent for
free processes (see \cite{Biane}).

The contribution of this paper is two-fold: noncommutative and
classical aspects. First, on the noncommutative side, we examine the
weak convergence of the measures \[
\underbrace{D_{1/B_{n}}\mu\triangleright
D_{1/B_{n}}\mu\triangleright\cdots\triangleright
D_{1/B_{n}}\mu}_{n\:\text{times}},\] where $D_{1/B_{n}}\mu$ is the
dilation of $\mu$ by a factor of $B_{n}^{-1}>0$. This corresponds to
the uniform convergence of the functions $B_{n}^{-1}F^{\circ
n}\left(B_{n}z\right)$ on compact subsets of $\mathbb{C}^{+}$. This
pattern of convergence has been considered in our previous work
\cite{JCLevy}, where we initiated the investigation of strict
domains of attraction relative to $\triangleright$. We proved in
\cite{JCLevy} that a law has a non-empty strict domain of attraction
if and only if it is strictly stable. The current paper contributes
to this study by characterizing the strict domain of attraction for
a particular strictly stable law; namely, the standard arc-sine law
$\gamma$ whose density is $\pi^{-1}(2-x^{2})^{-1/2}$ on the interval
$(-\sqrt{2},\sqrt{2})$. Here we discover that the monotonic strict
domain of attraction of $\gamma$ coincides with the classical strict
domain of attraction of the standard normal law $\mathcal{N}$
(Theorem 3.1). As a consequence, the monotonic central limit theorem
(CLT) is equivalent to the classical CLT or to the free CLT.

In the same vein, we show that our CLT result can be applied to the
study of free L\'{e}vy processes of the second kind (for short,
FLP2). The second kind processes are less studied than are the first
kind in the literature, mostly because their existence is hard to
establish. (The only known examples of FLP2 to date are Biane's
$\triangleright$-strictly stable ones in \cite{Biane}.) In
particular, the complete description for the L\'{e}vy measure
associated to a FLP2, a question due to Biane \cite[Section
4.7]{Biane}, is still not available at this point. Following Biane's
question, we show in this paper that the L\'{e}vy measure and every marginal law of a FLP2 cannot have slowly varying truncated
variances, when one of the marginal laws is centered or does not
have a finite mean (Theorem 3.7). This implies that it is not
possible to construct a FLP2 with zero expectations and finite
variances. For general processes with non-zero means, we have a
law of large numbers (Theorem 3.6).

Secondly, on the classical side, we address the applications of our
methods to infinite ergodic theory for inner functions on
$\mathbb{C}^{+}$. When the measure $\mu$ is singular relative to
Lebesgue measure $\lambda$ on $\mathbb{R}$, the function $F$ is
inner and its boundary restriction
\[
T(x)=\lim_{y\rightarrow0^{+}}F(x+iy)\in\mathbb{R}\] is a measure
preserving transformation on the measure space
$(\mathbb{R},\mathcal{B},\lambda)$, where $\mathcal{B}$ is the
$\sigma$-field of Borel measurable subsets of $\mathbb{R}$. A
famous example of this type is Boole's transformation:
$T(x)=x-x^{-1}$ on $\mathbb{R}$. The main result here is a simple condition for the conservativity of $T$ in probabilistic
terms (Theorem 3.9), which says that if $X_{1},X_{2},\cdots$ are
i.i.d. according to $\mu$ and $B_{n}^{-1}(X_{1}+X_{2}+\cdots+X_{n})$
tend weakly to the standard Gaussian with
$\sum_{n=1}^{\infty}B_{n}^{-2}=\infty$, then the transformation $T$
is conservative. This means that for any
$A\in\mathcal{B}$ with $\lambda(A)>0$, almost all points of $A$ will eventually return
to $A$ under the iteration of $T$; or, in the sense of
Poincar\'{e}'s recurrence theorem, that the relation
$\liminf_{n\rightarrow\infty}d(f(x),f(T^{\circ n}(x)))=0$ holds for
almost all $x\in \mathbb{R}$ and for any measurable $f$ taking
values in a separable metric space $(Y,d)$. Moreover, various
ergodic theorems now hold for the map $T$; for example, by
Hopf's ratio ergodic theorem, we obtain the $\lambda$-almost
everywhere convergence
\[ \lim_{n\rightarrow\infty}\frac{\sum_{j=0}^{n}f\circ T^{\circ
j}(x)}{\sum_{j=0}^{n}g\circ T^{\circ j}(x)}=\frac{\int_{X}f\,
d\lambda}{\int_{X}g\, d\lambda},\] whenever $f,g\in L^{1}(\lambda)$
and $g>0$ (cf. \cite{A}). Thus, we get an a.e. convergence result for the iteration
of $F$ on the boundary $\mathbb{R}$. Using Theorem 3.9, we also
construct a new class of conservative and ergodic measure preserving
transformations on $\mathbb{R}$, extending an old work of Aaronson
\cite{A1}. (See Example 3.10.)

Finally, we would like to comment on the methods used in our proofs.
The monotonic CLT for bounded summands was first shown by Muraki in \cite{MurakiCLT}, with a combinatorial proof (see also
\cite{Saigo}). Our results go beyond the case of bounded variables
and can treat variables with infinite variance. The proofs rely
solely on the free harmonic analysis tools as in
\cite{BVunbdd,PataCLT} and the theory of regular variation
\cite{BGT}, making no use of combinatorial methods. The key
ingredient here is a connection between the asymptotic behavior of
the norming constants $B_{n}$ and that of
the measure $\sigma$ appeared in the Nevanlinna form of the function
$F$ (see \eqref{eq:2.3}). Indeed, this consideration also plays an important role in our
construction of new conservative transformations. We would also like to mention that Anshelevich and Williams have recently proved a remarkable result on the equivalence between monotone and
Boolean limit theorems in \cite{AW}. Their technique relies on the Chernoff
product formula, which is different from the approach undertaken in
this paper. Here the equivalence between the classical and the
monotonic CLT's is proved directly, without any reference to
the Boolean theory.

This paper is organized into four sections. After collecting some
preliminary material in Section 2, we state our main results in
Section 3 and present their proofs in the last section.

\section{Setting and Basic Properties }

\subsection{Monotone convolution }

As shown by Franz in \cite{Franz}, given two measures
$\mu,\nu\in\mathcal{M}$ one can find two monotonically independent
random variables $X$ and $Y$ distributed according to $\mu$ and
$\nu$, respectively. The \emph{monotone convolution}
$\mu\triangleright\nu$ of the measures $\mu$ and $\nu$ is defined as
the distribution of $X+Y$. In the same paper, Franz also proved that
the definition of the measure $\mu\triangleright\nu$ does not depend on
the particular realization of the variables $X$ and $Y$. (We refer to \cite{Franz} for the
details of this construction.)

The calculation of monotone convolution of measures requires certain integral transforms, which we now review. First, the \emph{Cauchy transform} of a measure $\mu\in\mathcal{M}$ is
defined as \[ G_{\mu}(z)=\int_{-\infty}^{\infty}\frac{1}{z-t}\,
d\mu(t),\qquad z\in\mathbb{C}^{+},\] so that the reciprocal Cauchy
transform $F_{\mu}=1/G_{\mu}$ is an analytic self-map of
$\mathbb{C}^{+}$. The measure $\mu$ is completely determined by its
Cauchy transform $G_{\mu}$, and hence by the function $F_{\mu}$.
Given two measures $\mu,\nu\in\mathcal{M}$, it was shown in
\cite{Franz} that \[
F_{\mu\triangleright\nu}(z)=F_{\mu}\left(F_{\nu}(z)\right),\qquad
z\in\mathbb{C}^{+}.\] (See also \cite{BerMono} for measures with
bounded support.)

Let $j$ be a nonnegative integer. We write $\mu^{\triangleright j}$
for the \emph{$j$-th monotone convolution power
$\mu\triangleright\mu\triangleright\cdots\triangleright\mu$} of a
measure $\mu\in\mathcal{M}$, with
$\mu^{\triangleright0}=\delta_{0}$. Here the notation $\delta_{c}$
denotes the Dirac point mass in a point $c\in\mathbb{R}$.
Analogously, if $F$ is a map from a non-empty set $A$ into itself
then the notation $F^{\circ j}$ denotes its \emph{$j$-fold iterate}
$F\circ F\circ\cdots\circ F$, where the case $j=0$ means the
identity function on $A$.

For $\mu\in\mathcal{M}$, we denote by $D_{b}\mu$ the \emph{dilation}
of the measure $\mu$ by a factor $b>0$, that is,
$D_{b}\mu(A)=\mu(b^{-1}A)$ for all Borel subsets
$A\subset\mathbb{R}$. At the level of reciprocal Cauchy transforms,
this means that \[ F_{D_{b}\mu}(z)=bF_{\mu}(z/b),\qquad
z\in\mathbb{C}^{+}.\] Note that we have
$D_{b}(\mu\triangleright\nu)=D_{b}\mu\triangleright D_{b}\nu$ for
any $\mu,\nu\in\mathcal{M}$.

Given probability measures $\{\mu_{n}\}_{n=1}^{\infty}$ and $\mu$ on
$\mathbb{R}$, we say that $\mu_{n}$ \emph{converges} \emph{weakly}
to $\mu$, written as $\mu_{n}\Rightarrow\mu$, if \[
\lim_{n\rightarrow\infty}\int_{-\infty}^{\infty}f(t)\,
d\mu_{n}(t)=\int_{-\infty}^{\infty}f(t)\, d\mu(t)\] for every bounded, continuous real function $f$ on $\mathbb{R}$. Note that the weak convergence
$\mu_{n}\Rightarrow\mu$ holds if and only if the relation
$\lim_{n\rightarrow\infty}F_{\mu_{n}}(z)=F_{\mu}(z)$ holds for every
$z$ in $\mathbb{C}^{+}$ (cf. \cite{Geronimo}). Thus, if both
$\mu_{n}\Rightarrow\mu$ and $\nu_{n}\Rightarrow\nu$ hold, then one
has $\mu_{n}\triangleright\nu_{n}\Rightarrow\mu\triangleright\nu$.

The weak convergence of measures needs tightness. Recall that a
family $\mathcal{F}$ of positive Borel measures on $\mathbb{R}$ is
\emph{tight} if
\[
\lim_{y\rightarrow+\infty}\sup_{\mu\in\mathcal{F}}\mu(\{
t\in\mathbb{R}:\,\left|t\right|>y\})=0.\] We shall also mention that
the tightness for probability measures can be characterized through the asymptotics of their reciprocal Cauchy transforms. More
precisely, a family $\mathcal{F}\subset\mathcal{M}$ is tight if and
only if \begin{equation} F_{\mu}(iy)=iy(1+o(1)),\qquad
y>0,\label{eq:2.2}\end{equation} uniformly for $\mu\in\mathcal{F}$
as $y\rightarrow\infty$ (see \cite{BVunbdd} for the proof).

\subsection{Functions of slow variation}

Recall from \cite{BGT} that a positive Borel function $f$ on
$(0,\infty)$ is said to be \emph{regularly varying} if for every
constant $c>0$, one has \[
\lim_{x\rightarrow\infty}\frac{f(cx)}{f(x)}=c^{d}\] for some
$d\in\mathbb{R}$ ($d$ is called the \emph{index} \emph{of regular
variation}). A regularly varying function with index zero is said to
be \emph{slowly varying}. The notation $R_{d}$ denotes the class of
regularly varying functions with index $d$. 

The following properties of the class $R_{0}$ are important for our investigation. (See  the book \cite{BGT} for proofs.)

\begin{enumerate}
\item [(\textbf{P1}).] (Representation Theorem) A function $f:(0,\infty)\rightarrow(0,\infty)$
belongs to $R_{0}$ if and only if it is of the form \[
f(x)=c(x)\exp\left(\int_{1}^{x}\frac{\varepsilon(t)}{t}\, dt\right),\qquad x\geq1,\]
where $c(x)$ and $\varepsilon(x)$ are measurable and $c(x)\rightarrow c\in(0,\infty)$,
$\varepsilon(x)\rightarrow0$ as $x\rightarrow\infty$.
\item [(\textbf{P2}).] \textbf{}If $f\in R_{0}$, then for all $\varepsilon>0$
one has \[
\lim_{x\rightarrow\infty}x^{\varepsilon}f(x)=\infty\quad\text{and}\quad\lim_{x\rightarrow\infty}x^{-\varepsilon}f(x)=0.\]

\item [(\textbf{P3}).] \textbf{}If $f\in R_{0}$, then there is a positive sequence
$\{ B_{n}\}_{n=1}^{\infty}$ such that $\lim_{n\rightarrow\infty}B_{n}=\infty$
and the limit \[
\lim_{n\rightarrow\infty}nB_{n}^{-2}f(B_{n}x)=1\]
holds for each $x>0$.
\item [(\textbf{P4}).] (Monotone Equivalents) If $f\in R_{0}$ and $d>0$,
then there exists a non-decreasing function
$g:(0,\infty)\rightarrow(0,\infty)$ with $x^{d}f(x)\sim g(x)$ as
$x\rightarrow\infty$, that is,
$\lim_{x\rightarrow\infty}x^{d}f(x)/g(x)=1$.
\end{enumerate}

For a finite positive Borel measure $\mu$ on $\mathbb{R}$, we
introduce the functions
$H_{\mu},L_{\mu}:[0,\infty)\rightarrow[0,\infty)$ by
\[
H_{\mu}(x)=\int_{-x}^{x}t^{2}\, d\mu(t)\quad\text{and}\quad L_{\mu}(x)=\int_{-\infty}^{\infty}\frac{t^{2}x^{2}}{t^{2}+x^{2}}\, d\mu(t).\]
The mean and the second moment of $\mu$ are defined in the usual
way:\[
m(\mu)=\int_{-\infty}^{\infty}t\: d\mu(t)\quad\text{and}\quad m_{2}(\mu)=\int_{-\infty}^{\infty}t^{2}\: d\mu(t),\]
provided that the above integrals converge absolutely. (We also use
the somewhat abused notation $m(\mu)=\infty$ or $m_{2}(\mu)=\infty$
to indicate the divergence of these integrals.) The variance of $\mu$
will be written as $\text{var}(\mu)$.

Note that both $H_{\mu}$ and $L_{\mu}$ are continuous and
non-decreasing functions. Also, we have that
$H_{\mu}(x),L_{\mu}(x)>0$ for $x$ large enough if and only if
$\mu\neq r\delta_{0}$, $r\geq0$. Moreover, the functions
$H_{\mu}(x)$ and $L_{\mu}(x)$ are bounded if and only if the second
moment $m_{2}(\mu)$ exists, and in this case both functions tend to
$m_{2}(\mu)$ as $x\rightarrow\infty$. It is also known that if
$H_{\mu}$ varies slowly, then the mean $m(\mu)$ exists (see
\cite{Feller}, Section VIII.9.). We should also mention that
$H_{\mu}\in R_{0}$ if and only if $L_{\mu}\in R_{0}$, and in this case
we have $H_{\mu}(x)\sim L_{\mu}(x)$ as $x\rightarrow\infty$ (see
Proposition 3.3 in \cite{PataCLT}).

Every analytic map from $\mathbb{C}^{+}$ to
$\mathbb{C}^{+}\cup\mathbb{R}$ has a unique \emph{Nevanlinna
representation}  \cite[Theorem 6.2.1]{A}. In particular, the
reciprocal Cauchy transform $F_{\mu}$ of a measure
$\mu\in\mathcal{M}$ can be written as:
\begin{equation}
F_{\mu}(z)=z+a+\int_{-\infty}^{\infty}\frac{1+tz}{t-z}\,
d\sigma(t),\qquad z\in\mathbb{C}^{+},\label{eq:2.3}\end{equation}
where $a\in\mathbb{R}$ and $\sigma$ is a finite, positive Borel
measure on $\mathbb{R}$. This integral formula implies that $\Im
F_{\mu}(z)\geq\Im z$. Moreover, this inequality is strict for every
$z\in\mathbb{C}^{+}$ unless the measure $\mu$ is \emph{degenerate},
i.e., $\mu$ is a point mass. It is easy to verify that the monotone
convolution $\mu\triangleright\nu$ is always \emph{nondegenerate} if
$\mu$ or $\nu$ is not degenerate. Also, for an analytic map
$F:\mathbb{C}^{+}\rightarrow\mathbb{C}^{+}$ with the property
$\lim_{y\rightarrow\infty}F(iy)/iy=1$, the Nevanlinna form of $-1/F$
shows that the function $F$ is of the form $F=F_{\mu}$ for a unique
$\mu\in\mathcal{M}$.

The following result is a summary of Lemma 3.5 and Propositions 3.6
and 3.7 in \cite{PataCLT}, and it will be used repeatedly in this paper.

\begin{prop}
\cite{PataCLT} Let $\mu$ be a nondegenerate probability measure
on $\mathbb{R}$ whose reciprocal Cauchy transform $F_{\mu}$ is given
by \eqref{eq:2.3}. Then:
\begin{enumerate}
\item $H_{\mu}\in R_{0}$ if and only if $L_{\sigma}\in R_{0}$ or $L_{\sigma}(x)=0$
for every $x\geq0$. In this case we have the mean $m(\mu)=m(\sigma)-a$
and \[
H_{\mu}(x)-m(\mu)^{2}\sim L_{\sigma}(x)+\sigma(\mathbb{R})\qquad(x\rightarrow\infty).\]

\item If $L_{\sigma}\in R_{0}$, then it follows that \[
\lim_{x\rightarrow\infty}\frac{1}{L_{\sigma}(x)}\int_{-\infty}^{\infty}\frac{\left|t\right|^{3}x}{t^{2}+x^{2}}\, d\sigma(t)=0.\]

\item The measure $\mu$ has finite variance if and only if the measure
$\sigma$ does. In this case the variance of $\mu$ is equal to $m_{2}(\sigma)+\sigma(\mathbb{R})$.
\end{enumerate}
\end{prop}

\section{Main Results }

\subsection{Central limit theorems}

Let $\{ X_{n}\}_{n=1}^{\infty}$ be a sequence of classically
independent real-valued random variables with common distribution $\mu\in\mathcal{M}$. The distribution of the scaled sum
\[ \frac{X_{1}+X_{2}+\cdots+X_{n}}{C_{n}}\] is the probability
law $D_{1/C_{n}}\mu^{*n}$, where $\mu^{*n}$ denotes the $n$-fold
classical convolution power of the measure $\mu$. The classical CLT
asserts that there exists a sequence $C_{n}>0$ such that the
sequence $D_{1/C_{n}}\mu^{*n}$ converges weakly to the standard
normal law $\mathcal{N}$ as $n\rightarrow\infty$ if and only if
$H_{\mu}\in R_{0}$ and the mean $m(\mu)=0$ (see \cite{Feller}). Our
first result here is a monotonic analogue of the above CLT, in which
the limiting distribution is the standard arc-sine law $\gamma$ with
the reciprocal Cauchy transform\[
F_{\gamma}(z)=\sqrt{z^{2}-2},\qquad z\in\mathbb{C}^{+}.\] Here, and
in the sequel, the branch of the square root function is chosen so
that it is analytic in $\mathbb{C}\setminus[0,+\infty)$ and
$\sqrt{-1}=i$.

\begin{thm}
\textup{(General Monotone CLT)} Let $\mu$ be a nondegenerate probability
measure on $\mathbb{R}$. Then the following assertions are equivalent:
\begin{enumerate}
\item There exists a sequence $B_{n}>0$ such that the measures $D_{1/B_{n}}\mu^{\triangleright n}$
converge weakly to the arc-sine law $\gamma$ as $n\rightarrow\infty$.
\item The function $H_{\mu}$ is slowly varying and the mean of the measure
$\mu$ is zero.
\end{enumerate}
\end{thm}
A priori, the norming constants $B_{n}$ and $C_{n}$ in the monotonic
and the classical CLT's could be different. Here we would like to
emphasize that we can actually choose the same constants for both
limit theorems. More precisely, in the proof of Theorem 3.1 we take
$B_{n}=C_{n}$ to be the classical cutoff constants
\begin{equation} \inf\left\{ y>0:\, nH_{\mu}(y)\leq y^{2}\right\}
.\label{eq:3.1}\end{equation} (See \eqref{eq:4.1} below for
details.)

Let $\nu\in\mathcal{M}$ with $\nu\neq\delta_{0}$. We say that a
probability measure $\mu$ belongs to the \emph{strict domain of
attraction} of the law $\nu$ (relative to $\triangleright$, and we
write $\mu\in\mathcal{D}_{\triangleright}[\nu]$) if the weak
convergence $D_{1/B_{n}}\mu^{\triangleright n}\Rightarrow\nu$ holds
for some $B_{n}>0$. The strict domains of attraction relative to the
convolutions $*$ and $\boxplus$ are defined analogously.

For any probability law $\mu$, Theorem 3.1 shows the equivalence of
$D_{1/B_{n}}\mu^{\triangleright n}\Rightarrow\gamma$ and
$D_{1/B_{n}}\mu^{*n}\Rightarrow\mathcal{N}$. By the free CLT
\cite{PataCLT}, this equivalence extends to
$D_{1/B_{n}}\mu^{\boxplus n}\Rightarrow\mathcal{S}$, where
$\mathcal{S}$ denotes the standard semicircle law. We record this
consequence formally in the following

\begin{cor}
One has that $\mathcal{D}_{\triangleright}[\gamma]=\mathcal{D}_{*}[\mathcal{N}]=\mathcal{D}_{\boxplus}[\mathcal{S}]$.
\end{cor}
\begin{rem}
Given a measure $\mu\in\mathcal{M}$ and a sequence $B_{n}>0$, it was
shown in Theorem 4.3 of \cite{JCLevy} that if the measures
$D_{1/B_{n}}\mu^{\triangleright n}$ converge weakly to a
nondegenerate law $\nu\in\mathcal{M}$, then there exists a unique
$\alpha\in(0,2]$ such that the norming sequence
$B_{n}=n^{1/\alpha}f(n)$ for some $f\in R_{0}$. In fact, the
correspondence between the function $f(x)$ and the sequence $B_{n}$
is given by $f(x)=[x]^{-1/\alpha}B_{[x]}$, where $[x]$ means the
integral part of $x$. Furthermore, by Theorem 3.4 of \cite{JCLevy},
the limit law $\nu$ must be $\triangleright$-strictly stable with
the index $\alpha$. The arc-sine law $\gamma$ represents the class
of strictly stable laws of index $2$ in monotone probability, as the
normal law does in classical probability.
\end{rem}
Thus, the sequence $B_{n}$ in Theorem 3.1 is necessarily of the form
$\sqrt{n}b_{n}$, where $b_{n}$ is a slowly varying sequence. The
usual form of the CLT corresponds to the case when $b_{n}$ is a constant
sequence, and we have

\begin{thm}
\textup{(Monotone CLT)} Let $\mu$ be any nondegenerate probability
measure on $\mathbb{R}$, and let $a\in\mathbb{R}$ and $b>0$. Then
the following statements are equivalent:
\begin{enumerate}
\item The weak convergence $D_{1/\sqrt{nb}}\left(\mu\triangleright\delta_{-a}\right)^{\triangleright n}\Rightarrow\gamma$
holds.
\item The measure $\mu$ has finite variance.
\end{enumerate}
If \textup{(1)} and \textup{(2)} are satisfied, then the constants
$a$ and $b$ can be chosen as $a=m(\mu)$ and $b=\emph{var}(\mu)$.

\end{thm}

The proof of our results will be presented in the next section. Here
we would like to illustrate its main idea through the following example.

\begin{example}
Suppose $\mu=(\delta_{0}+\delta_{1})/2$. By taking
$B_{n}=\sqrt{n}/2$ and $a=1/2$, we obtain that \[
F_{n}(z)=F_{D_{1/B_{n}}(\mu\triangleright\delta_{-a})}(z)=z-\frac{1}{nz},\qquad
z\in\mathbb{C}^{+},\] for every $n\geq1$. To see that the measures
\[
\mu_{n}=D_{1/B_{n}}\left(\mu\triangleright\delta_{-a}\right)^{\triangleright
n}=\left[D_{1/B_{n}}\left(\mu\triangleright\delta_{-a}\right)\right]^{\triangleright
n}\] converge weakly to the law $\gamma$, we need to show the
pointwise convergence of the iterations $F_{n}^{\circ n}(z)$ to the
function $\sqrt{z^{2}-2}$ in an appropriate domain. To overcome the
difficulty of computing the iteration of $F_{n}$, we introduce the
following conjugacy functions: $\psi_{1}(z)=z^{2}$ and
$\psi_{2}(z)=\sqrt{z}$. Note that $\psi_{2}(-y^{2})=iy$ for all
$y>0$ and $\psi_{2}\circ\psi_{1}(z)=z$ for every
$z\in\mathbb{C}^{+}$. For $z=-y^{2}$, $y>1$, we observe that
\begin{eqnarray*}
F_{\mu_{n}}(\sqrt{z})^{2} & = & \psi_{1}\circ F_{n}^{\circ n}\circ\psi_{2}(z)\\
 & = & \left(\psi_{1}\circ F_{n}\circ\psi_{2}\right)^{\circ n}(z)\\
 & = & \left(z-\frac{2}{n}+\frac{1}{n^{2}z}\right)^{\circ n}\\
 & = & z-2+\frac{1}{n^{2}}\sum_{j=0}^{n-1}\frac{1}{\left(\psi_{1}\circ F_{n}\circ\psi_{2}\right)^{\circ j}(z)}=z-2+O\left(\frac{1}{n}\right).\end{eqnarray*}
Hence we have $F_{\mu_{n}}(iy)=\sqrt{(iy)^{2}-2+O(1/n)}$ for $y>1$,
which implies that
$\lim_{n\rightarrow\infty}F_{\mu_{n}}(z)=\sqrt{z^{2}-2}$ for $z=iy$,
$y>1$, as desired.
\end{example}

\subsection{Applications to free processes}

Let us now consider a \emph{$\triangleright$-convolution semigroup}
$\{\mu_{t}:t\geq0\}$ of probability measures on $\mathbb{R}$, that
is, $\mu_{0}=\delta_{0}$ and other $\mu_{t}$'s are nondegenerate for
$t>0$, $\mu_{s}\triangleright\mu_{t}=\mu_{s+t}$ for all $s,t\geq0$,
and the map $t\mapsto\mu_{t}$ is weakly continuous. The CLT holds in
this case, as well as the law of large numbers.

\begin{thm}
Let $\{\mu_{t}:t\geq0\}$ be a $\triangleright$-convolution semigroup.
\begin{enumerate}
\item If there is a time parameter $t_{0}>0$ such that $\mu_{t_{0}}\in\mathcal{D}_{\triangleright}[\gamma]$,
then one can find a positive function $B(t)\in R_{1/2}$ such that
$D_{1/B(t)}\mu_{t}\Rightarrow\gamma$ as $t\rightarrow\infty$. In
particular, we have $\mu_{t}\in\mathcal{D}_{\triangleright}[\gamma]$
for every $t>0$.
\item If the mean $a=m(\mu_{t_{0}})$ exists for some $t_{0}>0$, then $D_{1/t}\mu_{t}\Rightarrow\delta_{a/t_{0}}$
as $t\rightarrow\infty$.
\end{enumerate}
\end{thm}
A \emph{free additive process} (in law) is a family
$(Z_{t})_{t\geq0}$ of random variables with the following
properties: (i) For each $n\geq1$ and $0\leq
t_{0}<t_{1}<\cdots<t_{n}$, the increments\[
Z_{t_{0}},Z_{t_{1}}-Z_{t_{0}},Z_{t_{2}}-Z_{t_{1}},\cdots,Z_{t_{n}}-Z_{t_{n-1}}\]
are \emph{freely independent} in the sense of Voiculescu \cite{VKN}.
(ii) For any $t$ in $[0,\infty)$, the distribution of
$Z_{s+t}-Z_{t}$ converges weakly to $\delta_{0}$ as $s\rightarrow0$.
(iii) The distribution of $Z_{0}$ is $\delta_{0}$ and that of other
$Z_{t}$'s are nondegenerate.

Given such a process $(Z_{t})_{t\geq0}$, let $\mu_{t}$ be the distribution
of $Z_{t}$ and $\mu_{s,t}$ be the distributions of $Z_{t}-Z_{s}$
whenever $0\leq s<t$ . Clearly, these laws satisfy $\mu_{0}=\delta_{0}$,
\[
\mu_{t}=\mu_{s}\boxplus\mu_{s,t}\quad\text{and}\quad\mu_{s,u}=\mu_{s,t}\boxplus\mu_{t,u},\qquad
0\leq s<t<u,\] and $\mu_{t}\Rightarrow\delta_{0}$ as
$t\rightarrow0$. Conversely, given family
$\{\mu_{t}:t\geq0\}\cup\{\mu_{s,t}:0\leq s<t\}$ in $\mathcal{M}$
with the above properties, there exists a free additive process
$(Z_{t})_{t\geq0}$ such that the distributions of $Z_{t}$ and
$Z_{t}-Z_{s}$ are $\mu_{t}$ and $\mu_{s,t}$, respectively (see
\cite{Biane}).

For a free additive process $(Z_{t})_{t\geq0}$ with marginal laws
$\mu_{t}$, Biane's subordination result shows that there exist
unique measures $\sigma_{s,t}\in\mathcal{M}$ such that
$\mu_{t}=\mu_{s}\triangleright\sigma_{s,t}$ and
$\sigma_{s,t}\triangleright\sigma_{t,u}=\sigma_{s,u}$ for all $0\leq
s<t<u$.

A free additive process $(Z_{t})_{t\geq0}$ is said to be a
\emph{free L\'{e}vy process of the second kind} (FLP2) if
$\sigma_{s+h,t+h}=\sigma_{s,t}$ for all $0\leq s<t$ and $h\geq0$.
Thus, the marginal laws $\mu_{t}=\sigma_{0,t}$ of such a process
form a $\triangleright$-convolution semigroup, and their reciprocal
Cauchy transforms $F_{\mu_{t}}$ form a \emph{composition semigroup}
of analytic self-maps on $\mathbb{C}^{+}$.

It is well-known that for a composition semigroup $\{
F_{t}\}_{t\geq0}$ of analytic self-maps on $\mathbb{C}^{+}$ with
$F_{0}(z)=z$, the infinitesimal generator \begin{equation}
\varphi(z)=\lim_{\varepsilon\rightarrow0^{+}}\frac{F_{\varepsilon}(z)-z}{\varepsilon},\qquad
z\in\mathbb{C}^{+},\label{eq:3.2}\end{equation} of $\{
F_{t}\}_{t\geq0}$ exists and is unique \cite{BPsemigroup}. The
function
$\varphi:\mathbb{C}^{+}\rightarrow\mathbb{C}^{+}\cup\mathbb{R}$ is
analytic with the property
$\lim_{y\rightarrow\infty}\varphi(iy)/iy=0$, and hence it can be
written as \begin{equation}
\varphi(z)=a+\int_{-\infty}^{\infty}\frac{1+xz}{x-z}\,
d\rho(x).\label{eq:3.3}\end{equation}

In \cite{Biane} Biane showed that a measure $\rho$ corresponds to
the semigroup $\{ F_{t}\}_{t\geq0}$ associated with a FLP2 if and
only if for each $t>0$ the function $\varphi\circ{F}^{-1}_{t}$ has
an analytic continuation to $\mathbb{C}^{+}$, with values in
$\mathbb{C}^{+}$. He called such a measure $\rho$ the \emph{L\'{e}vy
measure} of FLP2 and raised the question of finding a direct
description for $\rho$.

We have the following result.

\begin{thm}
Let $(Z_{t})_{t\geq0}$ be a \emph{FLP2} with marginal laws
$\mu_{t}$, and let $\rho$ be the L\'{e}vy \emph{}measure of the
semigroup $F_{\mu_{t}}$. Suppose that there is a time parameter
$t_{0}>0$ such that the mean $m(\mu_{t_{0}})=0$ or
$m(\mu_{t_{0}})=\infty$. Then
\begin{enumerate}
\item For every $t>0$, the function $H_{\mu_{t}}$ is not slowly varying.
\item The function $H_{\rho}$ is not slowly varying.
\end{enumerate}
In particular, we have the second moments $m_{2}(\rho)=\infty$ and
$m_{2}(\mu_{t})=\infty$ for $t>0$.

\end{thm}
\begin{rem}
Every marginal law $\mu_{t}$ in a free additive process can be
written as a free convolution of infinitesimal probability measures;
for instance, we have
\[
\mu_{t}=\mu_{0,t/n}\boxplus\mu_{t/n,2t/n}\boxplus\cdots\boxplus\mu_{(n-1)t/n,t}.\]
Here the infinitesimality of the array
$\{\mu_{kt/n,(k+1)t/n}:n\geq1,\,0\leq k\leq n-1\}$ is guaranteed by
the stochastic continuity of the process $(Z_{t})_{t\geq0}$ (cf.
Remark 5.5 of \cite{BSLevy}). It follows that each measure $\mu_{t}$
is $\boxplus$-infinitely divisible (cf. \cite{BP in free hincin}).
\end{rem}

\subsection{Applications to ergodic theory of inner functions }

The general framework for infinite ergodic theory consists of a
$\sigma$-finite measure space $(X,\mathcal{F},\nu)$, $\nu(X)\neq0$, and
a measure preserving transformation $T:X\rightarrow X$. Thus, the
map $T$ is measurable with respect to the $\sigma$-field
$\mathcal{F}$ and $\nu(T^{-1}A)=\nu(A)$ for every set $A\in\mathcal{F}$.
The notation $T^{-1}A$ means the pre-image $\{ x\in X:Tx\in A\}$, and we write inductively $T^{-n}A=T^{-1}(T^{-(n-1)}A)$ for
$n\geq2$. As usual, the map $T$ is said to be \emph{ergodic} if for every set $A\in \mathcal{F}$ such that $T^{-1}A=A$, either $\nu(A)=0$ or $\nu(X\setminus{A})=0$. 

The key to understanding the recurrence behavior of the map $T$ lies in the study of its
\emph{conservativity}, a notion that can be traced back to E. Hopf's early work \cite{Hopf}. We say that $T$ is \emph{conservative} if for every set
$W\in\mathcal{F}$ such that $\{ T^{-n}W\}_{n=0}^{\infty}$ are
pairwise disjoint, necessarily $\nu(W)=0$. For a conservative
dynamical system $(X,\mathcal{F},\nu, T)$ and a non-null set $A\in \mathcal{F}$, one
has the occupation time $\sum_{j=0}^{\infty}I_A\circ T^{\circ
j}(x)=\infty$ a.e. on $A$ (i.e., the trajectory $\{T^{\circ j}(x)\}_{j=0}^{\infty}$ returns to the set $A$ infinitely often). The concept of conservativity plays
no role in finite measure spaces; for if $\nu(X)<\infty$, then any
measure preserving map $T$ on $X$ will be conservative. We refer to the book of Aaronson \cite{A} for the basics of infinite
ergodic theory.

An \emph{inner function} on $\mathbb{C}^{+}$ is an analytic map
$F:\mathbb{C}^{+}\rightarrow\mathbb{C}^{+}$ for which the limits \[
T(x)=\lim_{y\rightarrow0^{+}}F(x+iy)\in\mathbb{R}\] exist for almost
every $x\in\mathbb{R}$, relative to Lebesgue measure $\lambda$ on
$\mathbb{R}$. The measurable map $T:\mathbb{R}\rightarrow\mathbb{R}$
(defined modulo nullsets) is called the \emph{boundary restriction}
of $F$ to $\mathbb{R}$.

For $\mu\in\mathcal{M}$, we recall that\[
F_{\mu}(z)=z+a+\int_{-\infty}^{\infty}\frac{1+tz}{t-z}\,
d\sigma(t).\] It is known that the function $F_{\mu}$ is inner if
and only if $\sigma$ is singular with respect to $\lambda$ (cf.
Chapter 6 of \cite{A}) Clearly, this happens if and only if $\mu$ is
singular. Moreover, for a singular measure $\mu$, Letac
\cite{Letac1} has shown that the boundary restriction $T$ of
$F_{\mu}$ is a measure preserving transformation of the measure
space $(\mathbb{R},\mathcal{B},\lambda)$, and hence is an object of
ergodic theory. (The symbol $\mathcal{B}$ here denotes the Borel $\sigma$-field on $\mathbb{R}$.)

We shall fix a singular measure $\mu$ in $\mathcal{M}$. The ergodic
theory for the inner function $F_{\mu}$ was studied thoroughly in Aaronson's work
\cite{A1} (see also \cite{A}), where he proved that the
conservativity of the boundary restriction $T$ implies
the ergodicity of $T$, and that $T$ is conservative if and only if
\begin{equation} \sum_{n=1}^{\infty}\Im\frac{-1}{F_{\mu}^{\circ
n}(z)}=\infty\label{eq:3.4}\end{equation} for some $z\in \mathbb{C}^{+}$. Moreover, this condition is independent of the choice of
$z$. With this criterion, Aaronson further showed
that if the measure $\mu$ is compactly supported and $m(\mu)=0$,
then $T$ is conservative. (See also \cite{Letac2} for a different
approach to this result.) In the case of unbounded support, he
proved that if $\mu$ is symmetric (i.e., $\mu(A)=\mu(-A)$ for all $A\in\mathcal{B}$) and the function $H_{\mu}$ is
regularly varying with index $d$ for $d\in(0,1]$, then $T$ is
conservative.

Our next result gives a probabilistic criterion for the
conservativity of $T$, in which the measure $\mu$ is not assumed to
be compactly supported or symmetric. Note that the condition
\eqref{eq:3.5} below does not involve the iterations of $F_{\mu}$.

\begin{thm}
Let $\mu$ be a singular probability measure in the set
$\mathcal{D}_{*}[\mathcal{N}]$, and let $\{ B_{n}\}_{n=1}^{\infty}$ be a
positive sequence for which the classical \emph{CLT} holds for
$\mu$. If \begin{equation}
\sum_{n=1}^{\infty}\frac{1}{B_{n}^{^{2}}}=\infty,\label{eq:3.5}\end{equation}
then the boundary restriction of $F_{\mu}$ is conservative (and
hence ergodic).
\end{thm}

We conclude this section by showing some examples of conservative
transformations. It is easy to see from Proposition 2.1 that a
measure $\mu\in\mathcal{M}$ has finite variance and $m(\mu)=0$ if
and only if the Nevanlinna form of $F_{\mu}$ can be rewritten as: \[
F_{\mu}(z)=z+\int_{-\infty}^{\infty}\frac{1}{t-z}\, d\rho(t),\]
where $\rho$ is a finite positive Borel measure on $\mathbb{R}$.
Moreover, one has $\rho(\mathbb{R})=\text{var}(\mu)$.

\begin{example}
(a) If the measure $\mu$ has finite variance, then Theorem 3.4
implies that $B_{n}\sim \sqrt{n\text{var}(\mu)}$ as
$n\rightarrow\infty$. So, the condition \eqref{eq:3.5} is always
satisfied in this case. In particular, the \emph{generalized Boole
transformation} \[ T(x)=x+\sum_{n=1}^{\infty}\frac{p_{n}}{t_{n}-x}\]
is conservative and ergodic, whenever $t_{n}\in\mathbb{R}$ and
$p_{n}>0$ are sequences such that
$\sum_{n=1}^{\infty}p_{n}<\infty$. Boole's original transformation
$x\mapsto x-x^{-1}$ was proved to be ergodic by Adler and Weiss in
\cite{AdlerWeiss}. In case $\{ t_{n}\}_{n=1}^{\infty}$ and $\{
p_{n}\}_{n=1}^{\infty}$ are finite sequences, the ergodicity of $T$
is due to Li and Schweiger \cite{LiSch}.

(b) ($d=0$) In the case of infinite variance, the simplest way to construct a conservative transformation is to discretize a law from the domain of attraction of the normal law $\mathcal{N}$. Let
$\nu$ be a probability measure with $m(\nu)=0$ and $\nu(\{
t\in\mathbb{R}:\left|t\right|>x\})=x^{-2}$. Then the measure $\nu$
has infinite variance and satisfies the classical CLT with the
norming constants $B_{n}=\sqrt{n\log n}$ (see \eqref{eq:3.1}).
Let $\sigma$ be the atomic probability measure drawn from the law $\nu$:
\[ \sigma=\sum_{k\in\mathbb{Z}}p_{k}\delta_{k},\] where
$p_{k}=\nu([k,k+1))$ for all $k\in\mathbb{Z}$. It follows that the
measure $\sigma$ also satisfies the classical CLT with the same
constants $B_{n}$.

Now, let $\mu$ be the probability measure defined via the formula:
\[
F_{\mu}(z)=z+\int_{-\infty}^{\infty}\frac{1+tz}{t-z}\, d\sigma(t).\]
Since $m(\sigma)=0$ and $L_{\sigma}\in R_{0}$, we know $m(\mu)=0$
and $H_{\mu}\in R_{0}$ from Proposition 2.1; in other words,
$\mu\in\mathcal{D}_{*}[\mathcal{N}]$.

We now claim that the same sequence $\{ B_{n}\}_{n=1}^{\infty}$
serves as the norming constants for the CLT of the measure $\mu$.
Indeed, the sequence $B_{n}$ satisfies
$nB_{n}^{-2}H_{\sigma}(B_{n})\sim1$ $(n\rightarrow\infty)$. This is
equivalent to the relation
\[
nB_{n}^{-2}\left[L_{\sigma}(B_{n})+\sigma\left(\mathbb{R}\right)\right]\sim1\qquad(n\rightarrow\infty),\]
which is exactly the criterion of selecting the norming constants in
the Monotone CLT for the measure $\mu$ (see \eqref{eq:4.1}). It
follows that the sequence $B_{n}$ can be used as the norming
constants in the classical CLT for $\mu$.

Therefore, by Theorem 3.9, the boundary restriction \[
T(x)=x+\sum_{k\in\mathbb{Z}}\frac{1+kx}{k-x}p_{k}\] is conservative
and ergodic.

\end{example}

\section{The Proofs}

\subsection{Proof of Theorem 3.1 }

Fix a nondegenerate measure $\mu\in\mathcal{M}$. Suppose that the
function $H_{\mu}$ is slowly varying (and hence $m(\mu)$ exists).
Assume further that $m(\mu)=0$. Let us first specify the positive
sequence $\{ B_{n}\}_{n=1}^{\infty}$ that will be used to prove the
weak convergence of the measures \[
\mu_{n}=D_{1/B_{n}}\mu^{\triangleright n},\qquad n\geq1.\]

By Proposition 2.1 (1), the function $F_{\mu}$ has the Nevanlinna
form \[ F_{\mu}(z)=z+\int_{-\infty}^{\infty}\frac{1+t^{2}}{t-z}\,
d\sigma(t),\qquad z\in\mathbb{C}^{+},\] where the function
$L_{\sigma}\in R_{0}$ or $L_{\sigma}=0$. The case $L_{\sigma}=0$
implies that the measure $\sigma$ takes the form $r\delta_{0}$ for
some $r>0$, and hence the measure
$\mu=(\delta_{-\sqrt{r}}+\delta_{\sqrt{r}})/2$. Then this case
reduces to Example 3.5. Thus, we may and do assume that
$L_{\sigma}\neq0$ and $L_{\sigma}\in R_{0}$. Since measures with
finite variance will be treated in Theorem 3.4, we confine ourselves
to the case of infinite variance; that is,
$L_{\sigma}(x)\rightarrow\infty$ as $x\rightarrow\infty$.

Since the function $L_{\sigma}(x)+\sigma(\mathbb{R})$ is slowly
varying, (\textbf{P3}) implies that there exists a sequence
$B_{n}>0$ such that $\lim_{n\rightarrow\infty}B_{n}=+\infty$ and the
relation
\begin{equation}
nB_{n}^{-2}\left[L_{\sigma}\left(B_{n}y\right)+\sigma\left(\mathbb{R}\right)\right]\sim1\qquad(n\rightarrow\infty)\label{eq:4.1}\end{equation}
holds for each $y>0$. (The constants $B_{n}$ as in \eqref{eq:3.1} do
satisfy these conditions, see Feller's book \cite[Section
IX.8.]{Feller}.)

For notational convenience, we set $F_{n}(z)=F_{D_{1/B_{n}}\mu}(z)$
and write

\[
F_{n}(z)=z+\frac{1}{B_{n}}\int_{-\infty}^{\infty}\frac{1+t^{2}}{t-B_{n}z}\, d\sigma(t),\qquad z\in\mathbb{C}^{+}.\]
For every $z=x+iy\in\mathbb{C}^{+}$ with $\left|x\right|<y$, one
has that\begin{eqnarray*}
\left|\frac{1}{B_{n}}\int_{-\infty}^{\infty}\frac{1+t^{2}}{t-B_{n}z}\, d\sigma(t)\right| & \leq & \frac{1}{B_{n}}\int_{-\infty}^{\infty}\frac{1+t^{2}}{\left|t-B_{n}iy\right|}\left|\frac{t-B_{n}iy}{t-B_{n}z}\right|\, d\sigma(t)\\
 & \leq & \frac{2}{B_{n}}\int_{-\infty}^{\infty}\frac{1+t^{2}}{\sqrt{t^{2}+B_{n}^{^{2}}y^{2}}}\, d\sigma(t)\\
 & \leq & \frac{2}{B_{n}}\int_{-\infty}^{\infty}\left[\frac{1}{B_{n}y}+\frac{t^{2}}{t^{2}+B_{n}^{^{2}}y^{2}}\sqrt{t^{2}+B_{n}^{^{2}}y^{2}}\right]\, d\sigma(t)\\
 & \leq & \frac{2}{B_{n}^{^{2}}y}\int_{-\infty}^{\infty}\left[1+\frac{t^{2}B_{n}y}{t^{2}+B_{n}^{^{2}}y^{2}}\left(\left|t\right|+B_{n}y\right)\right]\, d\sigma(t).\end{eqnarray*}
Hence we have, for such $z$'s and $n\geq1$, that \begin{equation}
\left|F_{n}(z)-z\right|\leq\frac{2}{B_{n}^{^{2}}\Im z}\left[L_{\sigma}\left(B_{n}\Im z\right)+\sigma(\mathbb{R})+\int_{-\infty}^{\infty}\frac{\left|t\right|^{3}B_{n}\Im z}{t^{2}+B_{n}^{^{2}}\left(\Im z\right)^{2}}\, d\sigma(t)\right].\label{eq:4.2}\end{equation}

\begin{lem}
There exists $N=N(\mu)>0$ such that if $n\geq N$, then the estimate
\begin{equation}
\left|F_{n}^{\circ
j}(iy)-iy\right|\leq\frac{10j}{n}\label{eq:4.3}\end{equation} holds
uniformly for $y>10$ and for any integer $j$ between $0$ and $n$. In
particular, the sequence $\{\mu_{n}\}_{n=1}^{\infty}$ is tight.
\end{lem}

\begin{proof}
First of all, by (\textbf{P1}), there exists $N_{1}=N_{1}(\mu)$ such
that the estimate \[
\frac{L_{\sigma}\left(B_{n}y\right)+\sigma\left(\mathbb{R}\right)}{L_{\sigma}\left(B_{n}\right)+\sigma\left(\mathbb{R}\right)}=\frac{c\left(B_{n}y\right)}{c\left(B_{n}\right)}\exp\left(\int_{B_{n}}^{B_{n}y}\frac{\varepsilon(t)}{t}\,
dt\right)\leq2\exp\left(\int_{B_{n}}^{B_{n}y}\frac{1}{t}\,
dt\right)=2y\] holds for any $y>10$ and $n\geq N_{1}$. Then
\eqref{eq:4.1} shows that there is further a $N_{2}>N_{1}$ so that
\[
L_{\sigma}\left(B_{n}\right)+\sigma\left(\mathbb{R}\right)\leq\frac{5}{4}\frac{B_{n}^{^{2}}}{n},\qquad
n\geq N_{2}.\] Finally, by Proposition 2.1 (2), we can find the
desired $N>N_{2}$ such that \[
\int_{-\infty}^{\infty}\frac{\left|t\right|^{3}B_{n}y}{t^{2}+B_{n}^{^{2}}y^{2}}\,
d\sigma(t)\leq L_{\sigma}\left(B_{n}y\right)\] for any $n\geq N$ and
for $y>10$.

Combining these inequalities with \eqref{eq:4.2}, we obtain that
\begin{equation}
\left|F_{n}(z)-z\right|\leq\frac{10}{n}\label{eq:4.4}\end{equation}
for any $n\geq N$ and for any $z$ in the truncated cone \[
\Gamma_{10}=\{ x+iy\in\mathbb{C}^{+}:\,\left|x\right|<y,\, y>10\}.\]
In particular, the complex numbers $F_{n}(iy)$ lie in the cone
$\Gamma_{10}$ for $n\geq N$ and $y>10$. For such $n$'s and $y$'s we
now make use of \eqref{eq:4.4} to get\[
\left|F_{n}^{\circ2}(iy)-iy\right|\leq\left|F_{n}\left(F_{n}(iy)\right)-F_{n}(iy)\right|+\left|F_{n}(iy)-iy\right|\leq\frac{20}{n},\]
which implies further that $F_{n}^{\circ2}(iy)\in\Gamma_{10}$.
Proceeding inductively, we obtain that \[ \left|F_{n}^{\circ
j}(iy)-iy\right|\leq\left|F_{n}\left(F_{n}^{\circ(j-1)}(iy)\right)-F_{n}^{\circ(j-1)}(iy)\right|+\left|F_{n}^{\circ(j-1)}(iy)-iy\right|\leq\frac{10j}{n}\]
for any integer $0\leq j\leq n$, whence the estimate \eqref{eq:4.3}
holds.

Finally, it follows from \eqref{eq:4.3} that \[
F_{\mu_{n}}(iy)=F_{n}^{\circ n}(iy)=iy(1+o(1))\] uniformly in $n$ as
$y\rightarrow\infty$, and this establishes the tightness of
$\{\mu_{n}\}_{n=1}^{\infty}$.
\end{proof}
Next, let us recall the conjugacy functions appeared in Example 3.5:
$\psi_{1}(z)=z^{2}$ and $\psi_{2}(z)=\sqrt{z}$. We write \[
\psi_{1}\circ F_{n}\circ\psi_{2}(z)=F_{n}(\sqrt{z})^{2}=z+R\left(\psi_{2}(z)\right),\qquad z\in\mathbb{C}\setminus[0,+\infty),\]
where the function $R:\,\mathbb{C}^{+}\rightarrow\mathbb{C}$ is given
by \[
R(w)=\frac{2w}{B_{n}}\int_{-\infty}^{\infty}\frac{1+t^{2}}{t-B_{n}w}\, d\sigma(t)+\left[\frac{1}{B_{n}}\int_{-\infty}^{\infty}\frac{1+t^{2}}{t-B_{n}w}\, d\sigma(t)\right]^{2}.\]

We require the following result.
\begin{lem}
We shall have \[
\lim_{n\rightarrow\infty}\sum_{j=0}^{n-1}R\left(F_{n}^{\circ j}(iy)\right)=-2,\qquad10<y<11.\]
\end{lem}
\begin{proof}
Fix $y\in(10,11)$. Denoting $w_{j}=F_{n}^{\circ j}(iy)$ for
$j=0,1,\cdots,n-1$, it follows from \eqref{eq:4.3} that every
$w_{j}$ is in the set $\Gamma=\{
u+iv:\,\left|u\right|<v,\,10<v<21\}$ whenever $n\geq N$. Moreover,
since $\Gamma\subset\Gamma_{10}$ the estimate \eqref{eq:4.4} shows
that \[
\sum_{j=0}^{n-1}\left[\frac{1}{B_{n}}\int_{-\infty}^{\infty}\frac{1+t^{2}}{t-B_{n}w_{j}}\,
d\sigma(t)\right]^{2}=n\cdot
O\left(\frac{1}{n^{2}}\right)=o(1)\qquad(n\rightarrow\infty).\]

Thus, we only need to prove that \[
\sum_{j=0}^{n-1}\frac{w_{j}}{B_{n}}\int_{-\infty}^{\infty}\frac{1+t^{2}}{t-B_{n}w_{j}}\, d\sigma(t)=-1+o(1)\qquad(n\rightarrow\infty).\]
By virtue of \eqref{eq:4.1}, this amounts to showing that \[
\sum_{j=0}^{n-1}\left[\frac{w_{j}}{B_{n}}\int_{-\infty}^{\infty}\frac{1+t^{2}}{t-B_{n}w_{j}}\, d\sigma(t)+\frac{1}{B_{n}^{^{2}}}\left(L_{\sigma}(B_{n}y)+\sigma\left(\mathbb{R}\right)\right)\right]=o(1)\]
as $n\rightarrow\infty$.

Note that\begin{align*}
 & \frac{w_{j}}{B_{n}}\int_{-\infty}^{\infty}\frac{1+t^{2}}{t-B_{n}w_{j}}\, d\sigma(t)+\frac{1}{B_{n}^{^{2}}}\left(L_{\sigma}(B_{n}y)+\sigma\left(\mathbb{R}\right)\right)\\
 & =\frac{1}{B_{n}^{^{2}}}\int_{-\infty}^{\infty}\left[\frac{(1+t^{2})B_{n}w_{j}}{t-B_{n}w_{j}}+\frac{t^{2}B_{n}^{^{2}}y^{2}}{t^{2}+B_{n}^{^{2}}y^{2}}+1\right]\, d\sigma(t)\\
 & =\frac{1}{B_{n}^{^{2}}}\int_{-\infty}^{\infty}\frac{t^{3}(1+B_{n}^{^{2}}y^{2})+t^{4}B_{n}w_{j}+tB_{n}^{^{2}}y^{2}}{(t^{2}+B_{n}^{^{2}}y^{2})(t-B_{n}w_{j})}\,
 d\sigma(t)\\
 & =\frac{1}{B_{n}^{^{2}}}\int_{-\infty}^{\infty}\frac{t^{3}B_{n}y}{t^{2}+B_{n}^{^{2}}y^{2}}\left[\frac{1}{B_{n}y(t-B_{n}w_{j})}+\frac{B_{n}y}{t-B_{n}w_{j}}+\frac{w_{j}t}{y(t-B_{n}w_{j})}\right]\, d\sigma(t)\\
 & \quad+\frac{1}{B_{n}^{^{2}}}\int_{-\infty}^{\infty}\frac{tB_{n}y}{t^{2}+B_{n}^{^{2}}y^{2}}\left[\frac{B_{n}y}{t-B_{n}w_{j}}\right]\, d\sigma(t).\end{align*}
Meanwhile, observe that \[
\left|\frac{B_{n}y}{t-B_{n}w_{j}}\right|\leq\frac{y}{\Im
w_{j}}=\frac{y}{\Im F_{n}^{\circ j}(iy)}\leq1\] and \[
\left|\frac{t}{t-B_{n}w_{j}}\right|\leq\sqrt{1+\left(\frac{\Re
w_{j}}{\Im w_{j}}\right)^{2}}<\sqrt{2},\] for any $n\geq N$, $0\leq
j\leq n-1$, and $t\in\mathbb{R}$. Thus, when $n$ is large enough
such that $B_{n}>1$, we have\begin{align*}
 & \sum_{j=0}^{n-1}\left|\frac{w_{j}}{B_{n}}\int_{-\infty}^{\infty}\frac{1+t^{2}}{t-B_{n}w_{j}}\, d\sigma(t)+\frac{1}{B_{n}^{^{2}}}\left(L_{\sigma}(B_{n}y)+\sigma\left(\mathbb{R}\right)\right)\right|\\
 & \leq\frac{7n}{B_{n}^{^{2}}}\int_{-\infty}^{\infty}\frac{\left|t\right|^{3}B_{n}y}{t^{2}+B_{n}^{^{2}}y^{2}}\, d\sigma(t)+\frac{n}{B_{n}^{^{2}}}\int_{-\infty}^{\infty}\frac{\left|t\right|B_{n}y}{t^{2}+B_{n}^{^{2}}y^{2}}\, d\sigma(t).\end{align*}

By \eqref{eq:4.1} and Proposition 2.1 (2), we deduce that \[
\frac{n}{B_{n}^{^{2}}}\int_{-\infty}^{\infty}\frac{\left|t\right|^{3}B_{n}y}{t^{2}+B_{n}^{^{2}}y^{2}}\,
d\sigma(t)=\left(\frac{L_{\sigma}(B_{n}y)}{L_{\sigma}(B_{n}y)+\sigma\left(\mathbb{R}\right)}\right)\cdot
o(1)=o(1)\] and \begin{eqnarray*}
\frac{n}{B_{n}^{^{2}}}\int_{-\infty}^{\infty}\frac{\left|t\right|B_{n}y}{t^{2}+B_{n}^{^{2}}y^{2}}\, d\sigma(t) & \leq & \frac{n}{B_{n}^{^{2}}}\frac{1}{B_{n}}\int_{-\infty}^{\infty}\left|t\right|\, d\sigma(t)\\
 & = & \left(\frac{1+o(1)}{L_{\sigma}(B_{n}y)+\sigma\left(\mathbb{R}\right)}\right)\cdot o(1)=o(1)\end{eqnarray*}
as $n\rightarrow\infty$, proving the lemma.
\end{proof}
We are now ready to prove one implication in Theorem 3.1.

\begin{proof}[Proof of Theorem 3.1 (2) implies (1)] For $z=-y^{2}$, $y>1$,
we have\begin{eqnarray*}
F_{\mu_{n}}(\sqrt{z})^{2} & = & \psi_{1}\circ F_{n}^{\circ n}\circ\psi_{2}(z)\\
 & = & \left(\psi_{1}\circ F_{n}\circ\psi_{2}\right)^{\circ n}(z)\\
 & = & \left(z+R\left(\psi_{2}(z)\right)\right)^{\circ n}\\
 & = & z+\sum_{j=0}^{n-1}R\left(F_{n}^{\circ j}\left(\sqrt{z}\right)\right).\end{eqnarray*}
Thus, Lemma 4.2 implies \[
\lim_{n\rightarrow\infty}F_{\mu_{n}}(iy)=\sqrt{(iy)^{2}-2},\qquad10<y<11.\]
Since $\{\mu_{n}\}_{n=1}^{\infty}$ is a tight sequence, the above
equation determines uniquely the limit function
$F_{\gamma}(z)=\sqrt{z^{2}-2}$, and hence determines uniquely the
weak limit $\gamma$ of the sequence $\{\mu_{n}\}_{n=1}^{\infty}$.
Therefore, the full sequence $\{\mu_{n}\}_{n=1}^{\infty}$ converges
to the law $\gamma$.
\end{proof}
We now consider the converse of the central limit theorem. Suppose
$\mu$ is a distribution in the set
$\mathcal{D}_{\triangleright}[\gamma]$, that is, there exist norming
constants $B_{n}>0$ for which the measures\[
\mu_{n}=D_{1/B_{n}}\mu^{\triangleright n},\qquad n\geq1,\] converge
weakly to the law $\gamma$ as $n\rightarrow\infty$. Clearly, the
measure $\mu$ must be nondegenerate. We shall prove that the
function $L_{\sigma}$ is slowly varying, for it follows from
Proposition 2.1 (1) that $H_{\mu}\in R_{0}$. \vspace{-0.1in}
\begin{proof}[Proof of Theorem 3.1 (1) implies (2)] By Remark
3.3, the sequence $B_{n}$ is of the form $B_{n}=\sqrt{n}f(n)$, where
$f$ is a slowly varying function on $(0,\infty)$. By (\textbf{P4}),
every regularly varying function with positive index is
asymptotically equivalent to a non-decreasing function at infinity.
Hence, replacing $B_{n}$ by its monotone equivalent if necessary, we
can assume that $\{ B_{n}\}_{n=1}^{\infty}$ is an increasing
sequence.

Let $\varepsilon\in(0,1/2)$ be arbitrary. The weakly convergent sequence
$\{\mu_{n}\}_{n=1}^{\infty}$ is tight. By \eqref{eq:2.2}, there
exists $\beta=\beta(\varepsilon,\mu)\geq1$ such that for every $y>\beta$,
we have \[
\left|\frac{1}{B_{j}}F_{\mu}^{\circ j}(iB_{j}y)-iy\right|=\left|F_{\mu_{j}}(iy)-iy\right|\leq\varepsilon y,\qquad j\geq1.\]
Then, since $\{ B_{n}\}_{n=1}^{\infty}$ is monotonic, we deduce,
for such $y$'s and any $n>1$, that \begin{equation}
\left|F_{\mu}^{\circ j}(iB_{n}y)-iB_{n}y\right|\leq\varepsilon B_{n}y,\qquad0\leq j\leq n-1.\label{eq:4.5}\end{equation}

We write the function $F_{\mu}$ in the form: $F_{\mu}(z)=z+a+A(z)$,
where \[ A(z)=\int_{-\infty}^{\infty}\frac{1+tz}{t-z}\,
d\sigma(t),\qquad z\in\mathbb{C}^{+}.\] It follows from
\eqref{eq:4.5} and a straightforward calculation that \[
\left|A\left(F_{\mu}^{\circ
j}(iB_{n}y)\right)-A(iB_{n}y)\right|\leq2\varepsilon\Im
A\left(F_{\mu}^{\circ j}(iB_{n}y)\right),\qquad0\leq j\leq n-1.\]
Since \[ B_{n}^{-1}\sum_{j=0}^{n-1}A\left(F_{\mu}^{\circ
j}(B_{n}z)\right)=F_{\mu_{n}}(z)-z-nB_{n}^{-1}a,\] we obtain
\begin{equation}
\left|F_{\mu_{n}}(iy)-iy-nB_{n}^{-1}a-nB_{n}^{-1}A(iB_{n}y)\right|\leq2\varepsilon\Im\left[F_{\mu_{n}}(iy)-iy\right]\label{eq:4.6}\end{equation}
for any $y>\beta$ and for $n\geq1$. This implies that
\begin{equation} (1-2\varepsilon)f_{n}(y)\leq
nyU_{\sigma}(B_{n}y)\leq(1+2\varepsilon)f_{n}(y),\label{eq:4.7}\end{equation}
where the functions $f_{n}$ and $U_{\sigma}$ are defined as \[
f_{n}(y)=\Im F_{\mu_{n}}(iy)-y,\qquad y>0,\] and \[
U_{\sigma}(x)=\int_{-\infty}^{\infty}\frac{1+t^{2}}{x^{2}+t^{2}}\,
d\sigma(t),\qquad x>0.\] We note for further reference that
$f_{n}(y)\rightarrow\sqrt{y^{2}+2}-y$ as $n\rightarrow\infty$ for
each $y>0$, and that the function $U_{\sigma}(x)$ is decreasing in
$x$. Also, both $f_{n}$ and $U_{\sigma}$ are positive functions
because $\mu$ is nondegenerate.

Observe that $x^{2}U_{\sigma}(x)\sim
L_{\sigma}(x)+\sigma(\mathbb{R})$ as $x\rightarrow\infty$.
Therefore, we need to show that the function $U_{\sigma}$ is
regularly varying with index $-2$; that is, for any fixed $c>0$ we
shall prove that \[
\lim_{x\rightarrow\infty}U_{\sigma}(x)^{-1}U_{\sigma}(cx)=c^{-2}.\]
We proceed as follows. First, since $B_{n}\leq B_{n+1}$, for any
large $x>0$ we can choose a positive integer $n=n(x,y)$ such that
\[ B_{n}y\leq x<B_{n+1}y.\] Moreover, the monotonicity of
$U_{\sigma}$ and \eqref{eq:4.7} imply
\[
\frac{1-2\varepsilon}{1+2\varepsilon}\left(\frac{n}{n+1}\right)\frac{f_{n+1}(cy)}{cf_{n}(y)}\leq\frac{U_{\sigma}(cx)}{U_{\sigma}(x)}\leq\frac{1+2\varepsilon}{1-2\varepsilon}\left(\frac{n+1}{n}\right)\frac{f_{n}(cy)}{cf_{n+1}(y)}.\]
Secondly, by letting $x\rightarrow\infty$ (hence
$n\rightarrow\infty$), we obtain that \[
\left(\frac{1-2\varepsilon}{1+2\varepsilon}\right)c^{-2}\left(\frac{y+\sqrt{y^{2}+2}}{y+\sqrt{y^{2}+2/c^{2}}}\right)\leq\liminf_{x\rightarrow\infty}\frac{U_{\sigma}(cx)}{U_{\sigma}(x)}\]
and \[
\limsup_{x\rightarrow\infty}\frac{U_{\sigma}(cx)}{U_{\sigma}(x)}\leq\left(\frac{1+2\varepsilon}{1-2\varepsilon}\right)c^{-2}\left(\frac{y+\sqrt{y^{2}+2}}{y+\sqrt{y^{2}+2/c^{2}}}\right)\]
for every $y>\beta$. Finally, by first letting $y\rightarrow\infty$
and then $\varepsilon\rightarrow0$, we have that \[
c^{-2}\leq\liminf_{x\rightarrow\infty}\frac{U_{\sigma}(cx)}{U_{\sigma}(x)}\leq\limsup_{x\rightarrow\infty}\frac{U_{\sigma}(cx)}{U_{\sigma}(x)}\leq
c^{-2},\] whence the desired result for $U_{\sigma}$ follows.
Consequently, we have $H_{\mu}\in R_{0}$.

The last thing which needs to be proved is that $m(\mu)=0$. To this
purpose, we examine the real part of $A(iB_{n}y)$ in \eqref{eq:4.6}
for $y=2\beta$. We obtain that \[
\frac{n}{B_{n}}\left|a-\int_{-\infty}^{\infty}\frac{(4B_{n}^{^{2}}\beta^{^{2}}-1)t}{4B_{n}^{^{2}}\beta^{^{2}}+t^{2}}\,
d\sigma(t)\right|\leq(1+2\varepsilon)\left|F_{\mu_{n}}(2\beta
i)-2\beta i\right|=O\left(\frac{1}{\beta}\right).\] Note that
$B_{n}=\sqrt{n}f(n)$ with $f\in R_{0}$. Since the function $f$ grows
slower than any power at infinity (see (\textbf{P2})), the above
estimate and the dominated convergence theorem imply $a=m(\sigma)$.
Hence, by Proposition 2.1 (1), the measure $\mu$ has zero
expectation.
\end{proof}

\subsection{Proof of Theorem 3.4}

By Proposition 2.1 (3), if the measure $\mu$ has finite variance,
say, $\text{var}(\mu)=b>0$, then the constants $B_{n}$ can be taken
as $B_{n}=\sqrt{n(m_{2}(\sigma)+\sigma(\mathbb{R}))}=\sqrt{nb}$ in
order to satisfy the condition \eqref{eq:4.1}. Since the measure
$\mu\triangleright\delta_{-a}$ is simply a translation of $\mu$, the
function $H_{\mu\triangleright\delta_{-a}}$ is slowly varying if and
only if the function $H_{\mu}$ is. Thus, the proof of Theorem 3.4 is
merely a word-for-word translation of the proof of Theorem 3.1; only
this time the key estimate \eqref{eq:4.2} and Lemma 4.1 are easier to obtain because $(1+t^{2})\, d\sigma(t)$ is a finite
measure.

\subsection{Proof of Theorem 3.6}

We first focus on the proof of Theorem 3.6 (1). Suppose that $\{\mu_{t}:t\geq0\}$
is a $\triangleright$-convolution semigroup, and that there is a
$t_{0}>0$ such that $\mu_{t_{0}}\in\mathcal{D}_{\triangleright}[\gamma]$.
We aim to construct a positive function $B\in R_{1/2}$ for which
the weak convergence $D_{1/B(t)}\mu_{t}\Rightarrow\gamma$ ($t\rightarrow\infty$)
holds.

To this purpose, we let $t_{n}$ be any positive sequence so that
$\lim_{n\rightarrow\infty}t_{n}=\infty$ and write \[
\nu=\mu_{t_{0}}\quad\text{and}\quad\nu_{n}=\mu_{t_{0}(t_{n}/t_{0}-[t_{n}/t_{0}])},\]
where $[x]$ denotes again the integral part of $x$. By the semigroup
property, we have \[
\mu_{t_{n}}=\mu_{t_{0}[t_{n}/t_{0}]+t_{0}(t_{n}/t_{0}-[t_{n}/t_{0}])}=\nu^{\triangleright[t_{n}/t_{0}]}\triangleright\nu_{n},\qquad n\geq1.\]

\begin{proof}[Proof of Theorem 3.6 (1)] Since
$\nu\in\mathcal{D}_{\triangleright}[\gamma]$, there exists a
positive sequence $C_{n}$ such that $D_{1/C_{n}}\nu^{\triangleright
n}\Rightarrow\gamma$ along the set of positive integers. By Remark
3.3, the function $C(x)=C_{[x]}$ belongs to the class $R_{1/2}$.

Let us define $B(x)=C(x)/\sqrt{t_{0}}$ for $x>0$. The function $B(x)$ is in the class $R_{1/2}$, and we write \begin{equation}
D_{1/B(t_{n})}\mu_{t_{n}}=\left(D_{B(t_{n}/t_{0})/B(t_{n})}D_{1/B(t_{n}/t_{0})}\nu^{\triangleright[t_{n}/t_{0}]}\right)\triangleright D_{1/B(t_{n})}\nu_{n}.\label{eq:4.8}\end{equation}
Since $\{\nu_{n}\}_{n=1}^{\infty}\subset\{\mu_{t}:0\leq t\leq t_{0}\}$,
the stochastic continuity of the semigroup $\{\mu_{t}:t\geq0\}$ implies
that the family $\{\nu_{n}\}_{n=1}^{\infty}$ is tight. Then (\textbf{P2})
shows that $D_{1/B(t_{n})}\nu_{n}\Rightarrow\delta_{0}$ as $n\rightarrow\infty$.
Now, the desired weak convergence $D_{1/B(t_{n})}\mu_{t_{n}}\Rightarrow\gamma$
follows from \eqref{eq:4.8} and \[
\lim_{n\rightarrow\infty}B(t_{n})^{-1}B(t_{n}/t_{0})=1/\sqrt{t_{0}}.\]

Finally, every $\mu_{t}$ belongs to the set $\mathcal{D}_{\triangleright}[\gamma]$
simply because $\mu_{t}^{\triangleright n}=\mu_{nt}$ for $n\geq1$.
\end{proof}
Theorem 3.6 (2) follows from a similar consideration based on the
law of large numbers\footnote{For a WLLN of non-identical summands, see the article arXiv:1304.1230 (added April 20, 2013).}  $D_{1/n}\nu^{\triangleright
n}\Rightarrow\delta_{a}$ (see Theorem 5.1 of \cite{JCLevy}). We omit
the details.

\subsection{Proof of Theorem 3.7}

Consider the marginal laws $\mu_{t}$ of a FLP2, and let $\rho$ be
the L\'{e}vy \emph{}measure of the corresponding semigroup $\{ F_{\mu_{t}}\}_{t\ge0}$.
Suppose there is a $t_{0}>0$ such that $m(\mu_{t_{0}})=0$ or $m(\mu_{t_{0}})=\infty$.

In the sequel we write $F_{t}=F_{\mu_{t}}$ for each $t\geq0$ and
denote by $\sigma_{t}$ the measure associated to the Nevanlinna
form of $F_{t}$.

By calculus, we have \begin{equation}
F_{t}(z)-z=\int_{0}^{t}\varphi\left(F_{s}(z)\right)\, ds,\qquad t\geq0,\; z\in\mathbb{C}^{+},\label{eq:4.9}\end{equation}
where the function $\varphi$ is the infinitesimal generator of $\{ F_{t}\}_{t\geq0}$
as in \eqref{eq:3.2}.

\begin{lem}
For any fixed $t>0$, we shall have \[ t^{-1}\left(\Im
F_{t}(iy)-y\right)\sim\int_{0}^{1}\Im\varphi\left(F_{s}(iy)\right)\,
ds\qquad(y\rightarrow\infty).\] When $t=0$, we have \[
\Im{\varphi(iy)}\sim\int_{0}^{1}\Im\varphi\left(F_{s}(iy)\right)\,
ds\qquad(y\rightarrow\infty).\]
\end{lem}
\begin{proof}
Let $t\geq 0$ be given. By a change of variable in \eqref{eq:4.9},
we need to show that
\[ \int_{0}^{1}\Im\varphi\left(F_{ts}(iy)\right)\,
ds\sim\int_{0}^{1}\Im\varphi\left(F_{s}(iy)\right)\,
ds\qquad(y\rightarrow\infty).\] Indeed, denoting
$d\nu(x)=(1+x^{2})\, d\rho(x)$ as in the \emph{}Nevanlinna
representation \eqref{eq:3.3} of the generator $\varphi$, the
estimate
\begin{eqnarray*}
\left|\Im\varphi(F_{ts}(iy))-\Im\varphi(F_{s}(iy))\right| & \leq & \int_{-\infty}^{\infty}\left|\Im\frac{1}{x-F_{ts}(iy)}-\Im\frac{1}{x-F_{s}(iy)}\right|\, d\nu(x)\\
 & \leq & \int_{-\infty}^{\infty}\frac{\left|F_{ts}(iy)-F_{s}(iy)\right|}{\left|x-F_{ts}(iy)\right|\left|x-F_{s}(iy)\right|}\, d\nu(x)\\
& \leq & \varepsilon(y)\Im\varphi(F_{s}(iy))\end{eqnarray*}
holds for all $s\in[0,1]$ and $y>0$, where the bound \[
\varepsilon(y)=\sup_{0\leq
s\leq1}\left[y^{-1}\left|F_{ts}(iy)-F_{s}(iy)\right|\left(1+y^{-1}\left|F_{ts}(iy)-F_{s}(iy)\right|\right)\right].\]
Also, since the family $\{\mu_{u}\}_{0\leq u\leq T}$ is tight for
any finite time $T>0$, \eqref{eq:2.2} implies \[
\lim_{y\rightarrow\infty}\varepsilon(y)=0.\] Therefore, by
integrating the above estimate with respect to $s$, we get \[
\left|\int_{0}^{1}\Im\varphi(F_{ts}(iy))\,
ds-\int_{0}^{1}\Im\varphi(F_{s}(iy))\,
ds\right|\leq\varepsilon(y)\int_{0}^{1}\Im\varphi(F_{s}(iy))\, ds,\]
whence the desired result follows.
\end{proof}
We now prove Theorem 3.7.

\begin{proof}[Proof of Theorem 3.7 (1)] Assume the contrary that we can find a
$t>0$ such that $H_{\mu_{t}}\in R_{0}$.

Proposition 2.1 then shows that $L_{\sigma_{t}}\in R_{0}$ or
$L_{\sigma_{t}}=0$. In addition, by Lemma 4.3, we have \[
t_{0}^{-1}\left(\Im F_{t_{0}}(iy)-y\right)\sim t^{-1}\left(\Im
F_{t}(iy)-y\right)\qquad(y\rightarrow\infty),\] or, in other words,
\[ \int_{-\infty}^{\infty}\frac{(1+x^{2})y^{2}}{x^{2}+y^{2}}\,
d\sigma_{t_{0}}(x)\sim\frac{t_{0}}{t}\int_{-\infty}^{\infty}\frac{(1+x^{2})y^{2}}{x^{2}+y^{2}}\,
d\sigma_{t}(x)\qquad(y\rightarrow\infty).\] If $L_{\sigma_{t}}=0$,
then the measure $\sigma_{t_{0}}$ has finite second moment. Hence,
$m_{2}(\mu_{t_{0}})<\infty$ by Proposition 2.1. If
$L_{\sigma_{t}}\in R_{0}$, then so does the function
$L_{\sigma_{t_{0}}}$. Thus, by Proposition 2.1 again, we have
$H_{\mu_{t_{0}}}\in R_{0}$.

Clearly, the above conclusion gives a contradiction in the case
$m(\mu_{t_{0}})=\infty$. On the other hand, if $m(\mu_{t_{0}})=0$,
then Theorems 3.1 and 3.6 imply that
$D_{1/B(s)}\mu_{s}\Rightarrow\gamma$ as $s\rightarrow\infty$ for
some function $B(s)>0$. We know from Remark 3.8 that every marginal
law in a free additive process is $\boxplus$-infinitely divisible.
So, in this case the law $\gamma$ must be $\boxplus$-infinitely
divisible for being a weak limit of $\boxplus$-infinitely divisible
measures. This is, however, a contradiction because one can verify
that the inverse of the function $F_{\gamma}$ (relative to
composition) cannot be extended analytically to $\mathbb{C}^{+}$
(cf. Theorem 5.10 of \cite{BVunbdd}). Therefore, none of the
functions $H_{\mu_{t}}$ shall be slowly varying, proving Theorem 3.7
(1).
\end{proof}
\begin{proof}[Proof of Theorem 3.7 (2)] The second part of Lemma 4.3
implies \[
\int_{-\infty}^{\infty}\frac{(1+x^{2})y^{2}}{x^{2}+y^{2}}\,
d\rho(x)\sim\int_{-\infty}^{\infty}\frac{(1+x^{2})y^{2}}{x^{2}+y^{2}}\,
d\sigma_{1}(x)\qquad(y\rightarrow\infty).\] Since the function
$H_{\mu_{1}}$ is not slowly varying, the function $L_{\sigma_{1}}$
is not slowly varying neither. Thus, the above asymptotic
equivalence shows that the function $H_{\rho}$ cannot be slowly
varying.
\end{proof}

\subsection{Proof of Theorem 3.9}

We begin by noticing that the measure $\mu$ is also in the set
$\mathcal{D}_{\triangleright}[\gamma]$, and that we may (and do)
assume $\mu_{n}=D_{1/B_{n}}\mu^{\triangleright n}\Rightarrow\gamma$.
In view of Aaronson's condition \eqref{eq:3.4}, we seek for a better
control on the summands $-\Im G_{\mu^{\triangleright n}}(z)$. We
will do this for $z=i$.

\begin{proof}[Proof of Theorem 3.9] Since $\gamma$ is Lebesgue
absolutely continuous, Theorem 3.1 implies \[
\lim_{n\rightarrow\infty}\mu_{n}([-1,1])=\gamma([-1,1])=1/2.\] In
other words, one has \[ \mu^{\triangleright
n}([-B_{n},B_{n}])\sim1/2\qquad(n\rightarrow\infty).\] Also, it is
easy to see that \begin{eqnarray*}
\sum_{n=1}^{\infty}\Im\frac{-1}{F_{\mu}^{\circ n}(i)} & = & \sum_{n=1}^{\infty}-\Im G_{\mu^{\triangleright n}}(i)\\
 & = & \sum_{n=1}^{\infty}\int_{-\infty}^{\infty}\frac{1}{1+t^{2}}\, d\mu^{\triangleright n}(t)\\
 & \geq & \sum_{n=1}^{\infty}\int_{\left|t\right|\leq B_{n}}\frac{1}{1+t^{2}}\, d\mu^{\triangleright n}(t)\geq\sum_{n=1}^{\infty}\frac{1}{1+B_{n}^{^{2}}}\mu^{\triangleright n}([-B_{n},B_{n}]).\end{eqnarray*}
Clearly, if the sequence $B_{n}^{-1}$ is not square summable, then
\[
\sum_{n=1}^{\infty}\Im\frac{-1}{F_{\mu}^{\circ n}(i)}=\infty.\]
Therefore, the boundary restriction of $F_{\mu}$ is conservative.
\end{proof}

\begin{acknowledgement*}
\emph{The author would like to thank Professors Hari Bercovici and
Richard Bradley for stimulating communication, especially to
Professor Bradley for a discussion on the norming constants in the
central limit theorem. He is also indebted to Professor Raj
Srinivasan for his support during the course of this investigation.
The author was supported by an NSERC Canada Discovery Grant and a
University of Saskatchewan New Faculty Startup Grant.}
\end{acknowledgement*}


\begin{thebibliography}{10}

\bibitem[1]{A1}J. Aaronson, \emph{Ergodic theory for inner functions of the upper half
plane}, Ann. Inst. H. Poincar\'{e} Sect. B \textbf{14} (1978),
233-253.

\bibitem[2]{A}---------, \emph{An introduction to infinite ergodic theory}, Mathematical Surveys and Monographs,
50, American Mathematical Society, Providence, RI, 1997.

\bibitem[3]{AdlerWeiss}R. L. Adler and B. Weiss, \emph{The ergodic infinite measure preserving transformation of Boole}, Israel J. Math. \textbf{16} (1973), 263-278.

\bibitem[4]{AW}M. Anshelevich and J. D. Williams, \emph{Limit theorems for monotonic convolution and the Chernoff product formula}, Int. Math. Res. Not. IMRN (2013). doi:10.1093/imrn/rnt018

\bibitem[5]{BSLevy}O. E. Barndorff-Nielsen and S. Thorbj\o rnsen, \emph{Self-decomposability
and L\'{e}vy processes in free probability}, Bernoulli \textbf{8}
(2002), no. 3, 323-366.

\bibitem[6]{BenGhoralSchurmann}A. Ben Ghorbal and M. Sch\"{u}rmann,
\emph{Non-commutative notions of stochastic independence}, Math. Proc.
Cambridge Philos. Soc. \textbf{133} (2002), 531-561.

\bibitem[7]{BerMono}H. Bercovici, \emph{A remark on monotonic convolution},
Infin. Dimens. Anal. Quantum Probab. Relat. Top. \textbf{8} (2005),
no. 1, 117-120.

\bibitem[8]{BP in free hincin}H. Bercovici and V. Pata, \emph{A free analogue of
Hin\v{c}in's characterization of infinite divisibility}, Proc. Amer.
Math. Soc. \textbf{128} (2000), no. 4, 1011-1015.

\bibitem[9]{BVunbdd}H. Bercovici and D. V. Voiculescu, \emph{Free
convolution of measures with unbounded support}, Indiana Univ. Math.
J. \textbf{42} (1993), no. 3, 733-773.

\bibitem[10]{BPsemigroup}E. Berkson and H. Porta, \emph{Semigroups
of analytic functions and composition operators}, Michigan Math. J.
\textbf{25} (1978), no. 1, 101-115.

\bibitem[11]{Biane}Ph. Biane, \emph{Processes with free increments},
Math. Z. \textbf{227} (1998), no. 1, 143-174.

\bibitem[12]{BGT}N. H. Bingham, C. M. Goldie, and J. L. Teugels, \emph{Regular
variation}, Cambridge University Press, Cambridge, 1987.

\bibitem[13]{Feller}W. Feller, \emph{An introduction to probability
theory and its applications, Vol. II}, John Wiley \& Sons, New York,
1971.

\bibitem[14]{Franz}U. Franz, \emph{Monotone and Boolean convolutions
for non-compactly supported probability measures}, Indiana Univ. Math.
J. \textbf{58} (2009), no. 3, 1151-1185.

\bibitem[15]{Geronimo}J. S. Geronimo and T. P. Hill, \emph{Necessary
and sufficient condition that the limit of Stieltjes transforms is
a Stieltjes transform}, J. Approx. Theory \textbf{121} (2003), 54-60.

\bibitem[16]{Hopf}E. Hopf, \emph{Zwei S\"{a}tze \"{u}ber den wahrscheinlichen Verlauf der Bewegungen dynamischer Systeme}, Math. Ann. \textbf{103}
(1930), no. 1, 710-719.

\bibitem[17]{Letac1}G. Letac, \emph{Which functions preserve Cauchy laws?}, Proc. Amer. Math. Soc. \textbf{67}
(1977), 277-286.

\bibitem[18]{Letac2}G. Letac and D. Malouche, \emph{The Markov chain associated to a Pick function}, Probab. Theory Related Fields \textbf{118}
(2000), 439-454.

\bibitem[19]{LiSch}T-Y. Li and F. Schweiger, \emph{The generalized Boole transformation is ergodic}, Manuscr. Math. \textbf{25}
(1978), 161-167.

\bibitem[20]{Murakipreprint}N. Muraki, \emph{Monotonic convolution
and monotonic L\'{e}vy-Hin\v{c}in formula}, preprint, 2000.

\bibitem[21]{MurakiCLT}---------, \emph{Monotonic independence, monotonic
central limit theorem and monotonic law of small numbers}, Infin.
Dimens. Anal. Quantum Probab. Relat. Top. \textbf{4} (2001), no. 1,
39-58.

\bibitem[22]{MurakiFiveInd}---------, \emph{The five independences
as natural products}, Infin. Dimens. Anal. Quantum Probab. Relat.
Top. \textbf{6} (2003), no. 3, 337-371.

\bibitem[23]{PataCLT}V. Pata, \emph{The central limit theorem for
free additive convolution}, J. Funct. Anal. \textbf{140} (1996), no.
2, 359-380.

\bibitem[24]{Saigo}H. Saigo, \emph{A simple proof for monotone CLT},
Infin. Dimens. Anal. Quantum Probab. Relat. Top. \textbf{13} (2010),
no. 2, 339-343.

\bibitem[25]{Schurmann}M. Sch\"{u}rmann, \emph{Direct sums of tensor
products and non-commutative independence}, J. Funct. Anal. \textbf{133}
(1995), 1-9.

\bibitem[26]{Speicher}R. Speicher, \emph{On universal products},
Fields Institute Communications, Vol. 12 (D. V. Voiculescu, editor),
Amer. Math. Soc., 1997, 257-266.

\bibitem[27]{VoiCoalgebra}D. V. Voiculescu, \emph{The coalgebra of
the free difference quotient and free probability}, Internat. Math.
Res. Notices (2000), no. 2, 79-106.

\bibitem[28]{VKN}D. V. Voiculescu, K. J. Dykema and A. Nica, \emph{Free
Random Variables}, CRM Monograph Series, Vol. 1, Amer. Math. Soc.
Rhode Island, 1992.

\bibitem[29]{JCLevy}J.-C. Wang, \emph{Strict limit types for monotone
convolution}, J. Funct. Anal. \textbf{262} (2012), no. 1, 35-58.
\end{thebibliography}
\end{document}